\documentclass[12pt,twoside,a4paper,reqno]{article}
\usepackage{amsmath,amsthm,amssymb,graphicx,mathrsfs,amsmath,cite,array,epsfig,graphicx,cases,subfig,url,color,setspace}
\usepackage{mathrsfs, array,tikz,skak,chessfss,relsize,bbding,wasysym}
\usepackage{epsfig,graphicx,cases,subfig,url,fullpage}
\usepackage{epsfig, graphicx,mathrsfs, epstopdf, algorithm,algorithmic,boxedminipage,caption,overpic,float,placeins}
\newfloat{algorithm}{t}{lop}
\newtheorem{thm}{Theorem}[section]
\newtheorem{cor}[thm]{Corollary}

\newtheorem{fact}[thm]{Fact}
\newtheorem{lemma}[thm]{Lemma}

\newtheorem{prop}[thm]{Proposition}
\newtheorem{remark}[thm]{Remark}
\newtheorem{defi}[thm]{Definition}
\newcommand\ER{{Erd\H{o}s-R\'{e}nyi }}

\usetikzlibrary{shapes.geometric}
\definecolor{light-gray}{gray}{0.5}

\usepackage{xypic}
\def\e{\mathbb{E}\, }
\def\R{\mathbb{R}}

\def\ks{k_S}
\def\kt{k_T}

\def\X{\mathbb{X}}

\def\P{\mathbb{P}}

\def\PG{\P_{\mathcal{G}}}
\def\PM{\P_{\mathcal{M}}}

\def\G{{\mathcal{G}(n,\ks,\kt)}}
\def\M{{\mathcal{M}(n,\ks,\kt)}}

\def\OmegaM{\Omega_\mathcal{M}}
\def\OmegaG{\Omega_\mathcal{G}}

\newcommand{\ignore}[1]{ }

\begin{document}

\title{An Independent Process Approximation to Sparse Random Graphs with a Prescribed Number of Edges and Triangles}
\author{Stephen DeSalvo and M. Puck Rombach \\}
\maketitle

\begin{abstract}
We prove a \emph{pre-asymptotic} bound on the total variation distance between the uniform distribution over two types of undirected graphs with $n$ nodes.  
One distribution places a prescribed number of $\kt$ triangles and $\ks$ edges not involved in a triangle independently and uniformly over all possibilities, and the other is the uniform distribution over simple graphs with exactly $\kt$ triangles and $\ks$ edges not involved in a triangle.  
As a corollary, for $k_S = o(n)$ and $k_T = o(n)$ as $n$ tends to infinity, the total variation distance tends to $0$,
at a rate that is given explicitly.  
Our main tool is Chen--Stein Poisson approximation, hence our bounds are explicit for all finite values of the parameters.  

\smallskip
\noindent \textbf{Keywords.} Poisson approximation, Random Graph Theory, Stein's method, Asymptotic Enumeration of Graphs.

\smallskip
\noindent \textbf{MSC classes:} 05C30, 05A16, 05A20, 60C05

\end{abstract}




Many real-world networks display a property called \emph{clustering} or \emph{transitivity}, a dependency in the edge probability between two nodes on the number of common neighbours. In a human social network, for example, if two individuals have one or more friends in common, they are more likely themselves to be friends. Other types of networks, such as transportation networks, might display a negative transitivity, meaning that the probability of two nodes sharing a link decreases with the number of common neighbours, because such links may be unnecessary. The concept of clustering in social networks was introduced in~\cite{watts1998collective}. Formally, the \emph{clustering coefficient} $C$ of a simple, undirected graph $G$ is defined as 
$$C(G)=\frac{\mbox{number of closed triplets of vertices in } G}{\mbox{number of connected triplets of vertices in } G},$$  
where a closed triplet is an ordered triplet of vertices $(i,j,k)$ such that the induced subgraph on $(i,j,k)$ is a triangle, and a connected triplet is an ordered triplet of vertices $(i,j,k)$ such that the induced subgraph on $(i,j,k)$ is connected. 

Many widely used models for real-world networks, such as random geometric graphs ~\cite{dall2002random}, random intersection graphs~\cite{bloznelis2013degree}, random key graphs~\cite{yagan2009random}, the small-world model~\cite{watts1998collective}, and preferential attachment model~\cite{barabasi1999emergence,bollobas2003mathematical}, naturally display some clustering, even though this is not an explicit parameter in their definitions. Regarding the latter, Bollob\'{a}s and Riordan show in~\cite{bollobas2003mathematical} that it is possible to achieve almost any clustering coefficient or number of triangles by varying the parameters of the preferential attachment model as proposed in~\cite{barabasi1999emergence}. Several similar models with explicit clustering parameter have also been proposed~\cite{klemm2002highly,holme2002growing,serrano2005tuning,bansal2008evolving}.

In~\cite{newman2009random}, Newman proposes a uniform random graph model with specified number of edges and triangles for each vertex. This is a natural extension of the so-called \emph{configuration model}~\cite{bollobas1980probabilistic}, in which each vertex is given a specified number of half-edges, which are then joined up uniformly at random. 
In this paper, we make an attempt to study models where we control both the number of edges and the number of triangles in a graph, and sample from all such possible structures. 

We will focus our attention on variations of the random graph model $\mathcal{G}(n,m)$. The more famous model $\mathcal{G}(n,p)$ is the most widely studied random graph model in probabilistic combinatorics. It was introduced by Gilbert~\cite{gilbert1959random} and developed by Erd\H{o}s and R\'{e}nyi to an extent that it is sometimes referred to as the \ER random graph model. Their work started with the introduction of a very similar model $\mathcal{G}(n,m)$~\cite{erdos1959random}. In the model $\mathcal{G}(n,p)$, there are $n$ vertices and every edge $(i,j)$ appears independently at random with probability $p=p(n)$. In  $\mathcal{G}(n,m)$, there are $n$ vertices and $m$ edges, and the set of vertex pairs that have edges is chosen uniformly at random from all $\binom{n(n-1)/2}{m}$ such sets. When $m=np$, these two models behave similarly in the limit of $n \to \infty$ for many properties of interest, and $\mathcal{G}(n,p)$ is used more commonly because the independence between the edges allows for easier analysis. In this present work, however, it turns out that $\mathcal{G}(n,m)$ is much easier to analyze. See remark \ref{rejects}.


Our main result is a total variation distance bound between the distributions of the sets of edges in two different graph models, both of which are a variation on $\mathcal{G}(n,m)$.  In the first model, we fix $\ks$ edges and $\kt$ triangles and place them \emph{independently} at random.  
The resulting graph assumes no interaction between the edges and triangles, so one can imagine single edges as blue and the edges of each triangle as red, and we only count a triangle if it consists of all red edges, and similarly we do not count the red edges of a triangle as single blue edges;  
we denote by $\mathcal{M}(n,\ks,\kt)$ the random graph model which assigns equal probability to all such graphs.  
In the other model, we consider the set of all \emph{simple} graphs on $n$ nodes with exactly $\ks$ edges not involved in a triangle and $\kt$ triangles; we denote by $\mathcal{G}(n,\ks,\kt)$ the random graph model which assigns equal probability to all such graphs.  
Theorem~\ref{dtv theorem} puts a maximum bound on the total variation distance between 
$\mathcal{M}(n,\ks,\kt)$ and 
$\mathcal{G}(n,\ks,\kt)$, 
which, asymptotically as the number of nodes $n$ tends to infinity, tends to zero for $\ks = o(n)$ and $\kt = o(n)$.

To compute the total variation distance bound, we apply the Poisson approximation approach in \cite{ArratiaGoldstein}.  Poisson approximation has a rich history, see \cite{barbour1992poisson} and the references therein.  In particular, the application of Stein's method to the Poisson distribution \cite{chen1975poisson}, known as Chen--Stein Poisson approximation, not only allows one to prove Poisson convergence for certain collections of dependent random variables, it also provides a pre-asymptotic bound, \emph{i.e.}, a bound which is explicit and absolute for all finite values of the parameters.

There are two dominant error terms in the approximation, which are the expected number of occurrences of clustering by three single edges to form an unintended triangle, and the expected number of occurrences of clustering by three triangles to form an extra, unintended triangle.
We also use Theorem~\ref{dtv theorem} to prove Proposition~\ref{counting Gnk}, which is a pre-asymptotic estimate for 
the normalizing constant of $\mathcal{G}(n,\ks,\kt).$




\def\X{\mathbb{X}}


\section{Random Graph Models}

Let  
\begin{enumerate}
\item $\mathcal{G}(n,m)$ denote graphs picked uniformly from the set of all simple graphs with exactly $n$ nodes and $m$ edges, with each of the $\binom{\binom{n}{2}}{m}$ graphs equally likely (the original Erd\H{o}s--R\'{e}nyi model~\cite{erdos1959random});
\item $\mathcal{M}(n,m)$ denote graphs picked uniformly from the set of all graphs with no self loops, but with multiple edges allowed, and with exactly $n$ nodes and $m$ edges, with each edge placed independently and uniformly at random from the $\binom{n}{2}$ possible edges, so that each of the $\frac{\binom{n}{2}^{m}}{m!}$ possible graphs are equally likely.  
\end{enumerate}


We now define the extensions of each of these sets in our context.  Rather than just fix the number of edges, we fix the number of edges not involved in a triangle and also fix the number of triangles.  Let
\begin{enumerate}
\item $\mathcal{G}(n,\ks,\kt)$ denote graphs picked uniformly from the set of all simple graphs with exactly $n$ nodes, exactly $\ks$ edges not part of any triangle, and exactly $\kt$ triangles;
\item $\mathcal{M}(n,\ks,\kt)$ denote graphs formed by placing $\ks$ edges independently and uniformly over all $\binom{n}{2}$ possible pairs of nodes, and placing $\kt$ triangles independently and uniformly over all $\binom{n}{3}$ possible triplets of nodes on a graph.  Again, multiple edges are allowed.
\end{enumerate}


Let $\OmegaM$ and $\PM$ denote the sample space and probability measure of $\M$, respectively, and similarly let $\OmegaG$ and $\PG$ denote the sample space and probability measure of $\G$, respectively.  
Since $\OmegaG \subset \OmegaM$ and our measures are uniform, 
we define the set $E$ to be such that 
\begin{equation}\label{dtv x12}
\PG = \P_{\mathcal{M}|E},
\end{equation}
i.e., the uniform distribution over elements in $\OmegaG$.  
Thus our interest is in the quality of approximation of $\P_{\mathcal{M}|E}$ using $\PM$.

We now describe the set $E$ in terms of random variables $X_1, \ldots, X_7,$ where $X_i : \OmegaM \to \mathbb{Z}$, $i=1,\ldots, 7$.
\begin{align*}
X_1 &:=  \#\{\text{triplets of single edges that form a triangle} \}, \\
X_2 &:=  \#\{\text{extra triangle by two single edges connecting to an edge of a triangle} \}, \\
X_3 &:=  \#\{\text{extra triangle by single edge connecting two nodes in two triangles}, \\
&  \qquad \ \  \text{touching in a third node} \}, \\
X_4 &:= \#\{\text{extra triangle formed by three intersecting triangles}\}, \\
X_5 &:= \#\{\text{double edge from a single edge on top of a triangle edge}\},\\
X_6 &:=  \#\{\text{double edges}\}, \\
X_7 &:=  \#\{\text{double triangles}\}.
\end{align*}
Examples are shown in Figure~\ref{bad events}.  
We have $E = \{ X_1 = X_2 = \ldots = X_7 = 0\}.$ 

A common measure of distance between two probability distributions is total variation distance.  For any $k>0$, given two $\mathbb{R}^k$--valued probability distributions $\mathcal{L}(X)$ and $\mathcal{L}(Y)$, the total variation distance between $\mathcal{L}(X)$ and $\mathcal{L}(Y)$ is denoted by $d_{TV}(\mathcal{L}(X), \mathcal{L}(Y))$, and it is defined as
\begin{equation}\label{dtv}
d_{TV}(\mathcal{L}(X), \mathcal{L}(Y)) = \sup_{A \subseteq \R^k} |P(X\in A) - P(Y\in A)|, 
\end{equation}
where $A$ is any Borel--measurable set.  

We now present our main theorem, which is a quantitative bound on the total variation distance between $\PM$ and $\PG$

\begin{thm}\label{dtv theorem}
For all $n \geq 3$, $k_S \geq 0$, $k_T \geq 0$, we have
\begin{equation}
(1-e^{-\lambda}) - d_0 \leq d_{TV}(\PG, \PM) \leq (1-e^{-\lambda}) + d_0,
\label{dtv bounds}
\end{equation}
where $\lambda = \sum_{i=1}^7 \lambda_i$ is given by Lemma~\ref{lambda} and $d_0$ is given by Lemma~\ref{W lemma}. 
\end{thm}

Note that Equation~\eqref{dtv bounds} is \emph{a hard inequality}, \emph{i.e.}, not asymptotic.  Now we specify the asymptotic range of parameter values for which Equation~\eqref{dtv bounds} tends to 0.  

\begin{cor}
As $n$ tends to infinity, we have
\[ d_0 = \Theta\left( \frac{(\ks+\kt)^4}{n^6} \right), \]
\[ \lambda = \Theta\left(\frac{(\ks+\kt)^3+\kt^2}{n^3}+\frac{\kt\ks+\ks^2}{n^2}\right). \]
Whence,
\[\text{$\ks = o(n)$ and $\kt = o(n)$} \iff d_{TV}(\PG, \PM)\sim \lambda \to 0. \]
\end{cor}


\ignore{
Note that Equation~\eqref{dtv} is an exact inequality, \emph{i.e.}, not an upper bound.  
 Of course, it is not obvious how one would obtain an exact expression for the right-hand-side of Equation~\eqref{dtv}, due to the intricate dependencies.  This is where we shall appeal to Chen--Stein Poisson approximation, which provides approximations to sums of dependent random variables and provides guaranteed error bounds.  
}
\begin{figure}
\begin{tabular}{ccc}
1 \xymatrix{ & \bullet \ar@{.}[dr]^s \ar@{.}[dl]_s& \\ \bullet \ar@{.}[rr]_s & & \bullet } &
2 \xymatrix{ & \bullet \ar@{-}[dr] \ar@{-}[dl] \ar@{.}[r]^s & \bullet \ar@{.}[d]^s \\ \bullet \ar@{-}[rr]_t & & \bullet }  &
3\xymatrix{ & \bullet \ar@{-}[dr] \ar@{-}[dl] \ar@{.}[rr]^s & & \bullet \ar@{-}[dr] \ar@{-}[dl] \\ \bullet \ar@{-}[rr]_t & & \bullet \ar@{-}[rr]_t & & \bullet }  \\
4 \xymatrix{& & \bullet \ar@{-}[dl] \ar@{-}[dr] \\ & \bullet \ar@{-}[dr] \ar@{-}[dl] \ar@{-}[rr]_t & & \bullet \ar@{-}[dr] \ar@{-}[dl] \\ \bullet \ar@{-}[rr]_t & & \bullet \ar@{-}[rr]_t & & \bullet }  &
5 \xymatrix{ & \bullet \ar@{-}[dr] \ar@{-}[dl] \ar@/^1pc/@{.}[dr]^s & \\ \bullet \ar@{-}[rr]_t & & \bullet }  &
6 \xymatrix{ \bullet \ar@/^1pc/@{.}[rr]^s\ar@/_1pc/@{.}[rr]^s & & \bullet} \\
7 \xymatrix{ & \bullet \ar@{-}[dr] \ar@{-}[dl]  \ar@/^1pc/@{-}[dr] \ar@/_1pc/@{-}[dl]  \\ \bullet \ar@{-}[rr]_t \ar@/_1pc/@{-}[rr]_t & & \bullet } 
\end{tabular}
\caption{Pictoral representation of the bad events in a random graph model}
\label{bad events}
\end{figure}
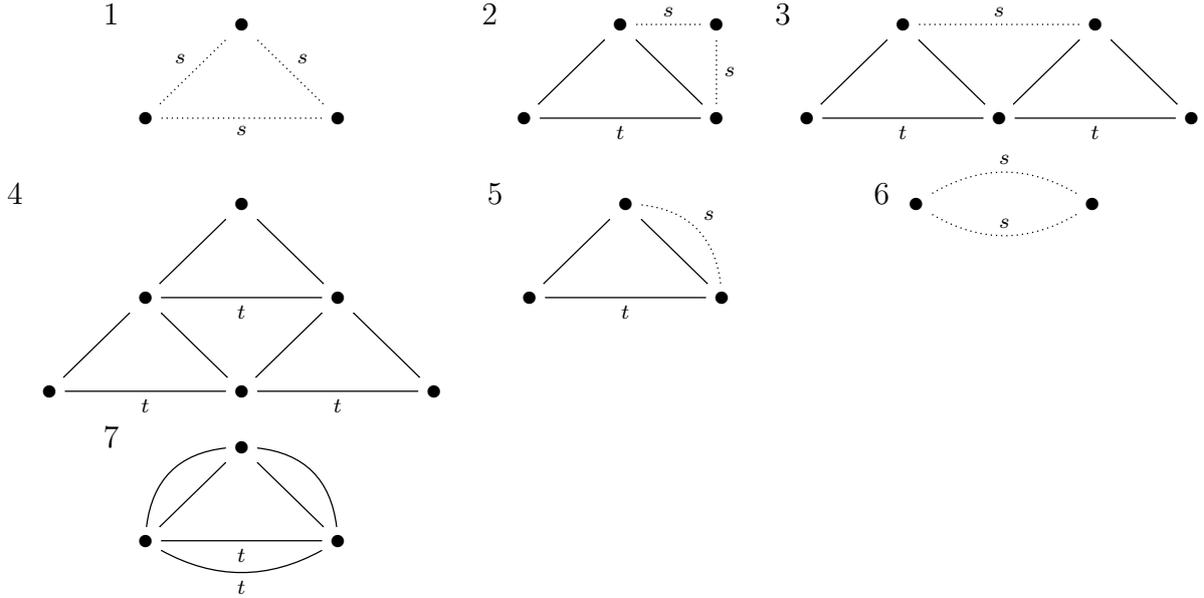

\begin{remark} \label{rejects}
{\rm
Let $\mathcal{G}(n,p)$ denote the random graph model consisting of the set of graphs with exactly $n$ nodes and each of the possible $\binom{n}{2}$ edges appearing independently with probability $p$, and also define its extension $\mathcal{G}(n,p_s, p_t)$, the set of graphs with exactly $n$ nodes and each of the possible $\binom{n}{2}$ edges appearing independently with probability $p_s$ and each of the possible $\binom{n}{3}$ triangles appearing independently with probability $p_t$.  
In this random graph model, the majority of the error in the approximation to 
$\mathcal{G}(n,\ks,\kt)$ is from the probability that the number of random edges/triangles is not \emph{exactly} the values specified.  This probability is given by a binomial distribution.  Let $S$ and $T$ denote the number of edges and triangles, respectively, in a random graph generated by $\mathcal{G}(n,p)$ or $\mathcal{G}(n,p_s, p_t)$.  

For $\mathcal{G}(n,p)$, with target $\ks$ specified in advance, we optimally choose $p_s = \ks/\binom{n}{2}$, which yields 
\[ \mathbb{P}(S = \ks) = \binom{\binom{n}{2}}{k_s}p_s^{k_s}(1-p_s)^{\binom{n}{2}-k_s} \sim \frac{1}{\sqrt{2\pi \ks\left(1-\ks/\binom{n}{2}\right)}}. \]
Similarly, for $\mathcal{G}(n,p_s,p_t)$, with targets $\ks$ and $\kt$ specified in advance, we optimally choose $p_s = \ks / \binom{n}{2}$ and $p_t = \kt/\binom{n}{3}$, which by independence yields
\[ \mathbb{P}(S = \ks, T = \kt) \sim \left(2\pi \sqrt{ \ks \kt \left(1-\frac{\ks}{\binom{n}{2}}\right)\left(1-\frac{\kt}{\binom{n}{3}}\right)}\right)^{-1}. \]

Denote by $\P_{p_s,p_t}$ the probaiblity measure in the 
random graph model $\mathcal{G}(n,p_s,p_t)$.  
Letting $A = \{S = \ks, T = \kt\}$ in Equation~\eqref{dtv}, we have
\[ d_{TV}(\PG, \P_{p_s,p_t}) \geq 1 - P(S = \ks, T=\kt). \]
These calculations demonstrate that total variation distance is too strong of a metric to be used for this type of approximation.  

}\end{remark}

\def\ks{k_S}
\def\kt{k_T}

\section{Proof of Theorem~\ref{dtv theorem}}
\label{proof}

First, we note that as a consequence of Equation~\eqref{dtv x12}, and using $A = E$ in Equation~\eqref{dtv}, we have
\[ d_{TV}(\PG, \PM) = 1-\P(X_1=\ldots=X_7=0) = 1-\P(W = 0), \]
where $W = \sum_{i=1}^7 X_i.$  
 Since the $X_i$, $i=1,\ldots, 7$ are not independent, we apply Chen--Stein Poisson approximation.  
Specifically, our proof is an application of Theorem~1 in \cite{ArratiaGoldstein}.  

We begin by defining the quantities $b_1$ and $b_2$, which are required to specify the upper bound in Theorem~\ref{dtv theorem}.  Suppose there is some countable or finite index set $I$.  For each $\alpha \in I$, let $Y_\alpha$ denote an indicator random variable, with $p_\alpha := \e Y_\alpha = P(Y_\alpha = 1) >0$, and $p_{\alpha\beta} := \e Y_\alpha Y_\beta$.  Define $W := \sum_{\alpha \in I} Y_\alpha$, and $\lambda := \e W = \sum_{\alpha\in I} p_\alpha$.  Next, for each $\alpha \in I$, we define a dependency neighborhood $B_\alpha$ which consists of all indices $\beta \in I$ for which $Y_\alpha$ and $Y_\beta$ are dependent.  Then we define the quantities
\begin{align}
\label{b1} b_1 :=&  \sum_{\alpha \in I} \sum_{\beta \in B_\alpha} p_\alpha p_\beta, \\
\label{b2} b_2 :=&  \sum_{\alpha \in I} \sum_{\alpha \neq \beta\in B_\alpha} p_{\alpha \beta}, \\
\label{b3} b_3 :=& \sum_{\alpha\in I} \e |\e\{Y_\alpha - p_\alpha \mid \sigma(Y_\beta: \beta \notin B_\alpha) \} |.
 \end{align}
Before one spends too much time parsing the precise meaning of $b_3$, we note that 
when dependency neighborhoods $B_\alpha$ are chosen so that $Y_\alpha$ and $Y_\beta$ are independent for $\beta \notin B_\alpha$, as in our setting, then $b_3 = 0$, 
and so it is just the first two quantities, $b_1$ and $b_2$, which need to be computed.  

\begin{thm}[\cite{ArratiaGoldstein}]
Let $Z$ denote a Poisson random variable, independent of $W$,  with $\e Z = \e W = \lambda$.  Then we have
\[ d_{TV}(\mathcal{L}(W),\mathcal{L}(Z)) \leq 2(b_1 + b_2 + b_3), \]
and
\[ |\P(W=0) - e^{-\lambda}| \leq \min(1,\lambda^{-1})(b_1+b_2+b_3). \]
\end{thm}
\ignore{
\begin{thm}[\cite{ArratiaGoldstein}]
Let $(X_\alpha),$ $\alpha\in I,$ denote a process of independent Bernoulli random variables with the same marginal distributions as $(Y_\alpha)$, $\alpha\in I$.  Then we have  
\[ d_{TV}((X_\alpha)_{\alpha\in I}, (Y_\alpha)_{\alpha\in I}) \leq 2(2b_1+2b_2+b_3)+2 \sum_\alpha p_\alpha^2. \]
\end{thm}
}

\ignore{
In order to apply this to our random graph model, we first note that for each triplet of single edges $\alpha \in \Gamma_1$, at most one of the collection $\{Y_{i,j,k}^{(a,b,c)}\}_{1\leq i<j<k\leq n}$ can be one.  Thus, we define 
\[ Y_\alpha :=  \sum_{1\leq i<j<k\leq n} Y_{i,j,k}^{(a,b,c)}, \qquad  \alpha \in \Gamma_1,
\]
which is also a Bernoulli random variable.  Similarly, we define for each $\alpha \in \Gamma_2$, 
\[Y_\alpha := \sum_{1 \leq i<j<k<\ell \leq n} Y_{i,j,k,\ell}^{(a,b),(u)}, \qquad \{(a,b),(u)\}=\alpha \in \Gamma_2.\]
We do the same for the remaining cases as well.
}

Denote the set of single edges by $K_s$, and the set of triangles by $K_t$.  We define the unordered set of triplets of single edges,
\[\Gamma_1 := \{ \{a,b,c\} : a,b,c \in K_s, \text{ distinct} \}.
\]
Then $X_1 = \sum_{\alpha \in K_s} Y_\alpha$, where $Y_\alpha$ is the indicator random variable that the three single edges in $\alpha$ form a triangle.   

Similarly, define
\[ \Gamma_2 := \{ \{a,b\}, \{\alpha\} : a,b\in K_s, \alpha\in K_t, \text{ distinct}\}. \]
Then, as before, $X_2 = \sum_{\alpha \in \Gamma_2} Y_{\alpha}$.  The other cases for $\Gamma_i$ and $X_i$, $i=3,4,5,6,7$, are defined similarly.  

Now we take $I = \bigcup_{i=1}^7 \Gamma_i$.  Then for each $\alpha \in I$, we let $B_\alpha$ denote the set of indices $\beta$ for which $Y_\alpha$ and $Y_\beta$ are independent; this is precisely the set of collections of indices where $\alpha$ and $\beta$ share any combination of \emph{at least 2} single edges or triangles.  

  For example, when $\alpha = \{a,b,c\} \in \Gamma_1$, then $B_\alpha$ is the set of all $Y_\beta$, $\beta \in I$ which contain at least two of $a$, $b$, or $c$.  For $\beta = \{ \{a,b\},\{u\}\}\in \Gamma_2$, we have $\beta \in B_\alpha$.  We also have $Y_\beta \notin B_\alpha$ for any $\alpha \in \Gamma_1$ and $\beta \in \Gamma_4$, since $\Gamma_1$ only consists of single edges and $\Gamma_4$ only consists of triangles.  
  
When $\alpha, \beta \in I$ do not share any elements, it is obvious that $Y_\alpha$ and $Y_\beta$ are independent.  Furthermore, even when $\alpha, \beta$ share exactly one element, $Y_\alpha$ and $Y_\beta$ are still independent, since conditioning on an occurrence of $Y_\alpha$ does not give any information regarding \emph{where} the occurrence occurred.  

It now remains to compute the desired quantities; we start with $b_1$.  By simple counting arguments, we have the following.

\begin{lemma}\label{ps}
\begin{equation}\label{pvals}
\large
p_\alpha = \left\{\begin{array}{cc}
	 \frac{\binom{n}{3}}{\binom{n}{2}^3}, & \alpha \in \Gamma_1, \\
	 \frac{\binom{n}{4} \binom{4}{3} \cdot 3}{\binom{n}{3} \binom{n}{2}^2}, & \alpha \in \Gamma_2,\\
	 \frac{\binom{n}{5} \binom{5}{3} \binom{3}{2}}{\binom{n}{3}^2 \binom{n}{2}}, & \alpha \in \Gamma_3, \\
	 \frac{\binom{n}{6} \binom{6}{3} \binom{3}{2}^2}{\binom{n}{3}^3}, & \alpha \in \Gamma_4, \\
	 \frac{\binom{n}{3}\cdot 3}{\binom{n}{3}\binom{n}{2}}, & \alpha \in \Gamma_5, \\
	 \frac{\binom{n}{2}}{\binom{n}{2}^2}, & \alpha \in \Gamma_6, \\
	 \frac{\binom{n}{3}}{\binom{n}{3}^2}, & \alpha \in \Gamma_7.
\end{array} \right\}
\end{equation}
\end{lemma}

Next, since the terms $p_\alpha$ and $p_\beta$ in the sum in Equation~\eqref{b1} do not depend on the particular set of indices, we simply need to count the number overlapping indices for which $Y_\alpha$ and $Y_\beta$ are dependent; when $\alpha \in \Gamma_i$ and $\beta \in \Gamma_j$, we denote the number of indices $\beta\in B_\alpha$ by $C_{i,j}$.  The lemma below follows by straightforward counting.  
\begin{lemma} \label{Cs}
We have 

\[
\begin{array}{llll}
C_{1,1} = \binom{\ks}{4}\binom{4}{2}, & C_{1,2} = \binom{\ks}{3} \binom{3}{2} \kt, & C_{1,3} = 0, & C_{1,4} = 0 \\ 
C_{1,5} = 0 & C_{1,6} = \binom{\ks}{3}\binom{3}{2} & C_{1,7} = 0 & C_{2,2} = \binom{\ks}{2}\binom{\kt}{2}+\binom{\ks}{3}\binom{3}{1}\kt \\
 C_{2,3} = \binom{2}{1}\binom{\ks}{2}\binom{2}{1}\binom{\kt}{2} & C_{2,4} = 0 & C_{2,5} = \binom{\ks}{2}\binom{2}{1}\kt & C_{2,6} = \binom{\ks}{2}\kt \\
 C_{2,7} = 0 & C_{3,3} = \binom{\kt}{2}\binom{\ks}{2}+\ks \binom{\kt}{3}\binom{3}{1} & C_{3,4} = \ks \binom{\kt}{3}\binom{3}{1} & C_{3,5} = \ks \binom{\kt}{2}\binom{2}{1} \\
 C_{3,6} = 0 & C_{3,7} = \binom{\ks}{1}\binom{\kt}{2} & C_{4,4} = \binom{\kt}{4}\binom{4}{2} & C_{4,5} = 0 \\
 C_{4,6} = 0 & C_{4,7} = \binom{\kt}{3}\binom{3}{2} & C_{5,5} = 0 & C_{5,6} = 0 \\
 C_{5,7} = 0 & C_{6,6} = 0 & C_{6,7} = 0 & C_{7,7} = 0.
\end{array}
\]
\end{lemma}

We can now state the formula for $b_1$ below.

\begin{prop}
For the random graph model $\mathcal{M}(n,k_s, k_t)$, we have
\begin{equation} \label{eq b1}
b_1 = \sum_{i=1}^7 \sum_{j=i}^7 C_{i,j}\, p_i\, p_j,
\end{equation}
where the values are specified in Lemmas~\ref{ps}, \ref{Cs}, and \ref{Gs}.
\ignore{
The values of $|B_{i,j}|$ are contained in the following table,  Many of the quantities appear multiple times and so we make the following definitions: $x := 2!\binom{k_s}{2}, y := 2! \binom{k_t}{2}$,
\[\tiny
\begin{array}{llllllll}
j\setminus i & 1 & 2 & 3 & 4 & 5 & 6 & 7 \\
1 & 3x + 3k_s + 3! & 3x + 6 k_s & 3 k_s & 0 & 3 k_s & 6 k_s & 0 \\
2 & 		  & k_t + 2 k_t k_s + 2 k_s + 2 & 2k_t + 4 k_t k_s + 2 k_s & 3 k_t & k_t + 2 k_s & 0 & 2 k_t \\
3 & 		  & 				    & k_s + 4 k_t + y + 2 k_s k_t + 2 & 3 k_t + 3 y & 2 k_t + k_s & 2 k_s & 2 k_t\\
4 & & & & 3 k_t + 3 y + 3! & 3 k_t & 0 & 6 k_t \\
5 & & & & & k_t + k_s + k_t k_s +1 & 2 k_s & 2 k_t  \\
6 & & & & & & 2 & 0\\
7 & & & & & & & 2, \\
\end{array}
\]
and $p_i$ is given in Equation~\eqref{pvals}.  
}
\end{prop}

\begin{cor}
Asymptotically, as $n\to\infty$, 
we have
\[ b_1  \sim O\left( \frac{(\ks+\kt)^4}{n^6}\right).  \]
\end{cor}

The expressions for $p_{\alpha\beta}$ are no more difficult to calculate, although care must be taken to account for all possible symmetries.  Since the values for $p_\alpha$, $\alpha \in I$ have already been specified, we instead focus on calculating $p_{\beta | \alpha}:=\e(Y_\beta | Y_\alpha)$, since $p_{\alpha\beta} = p_\alpha\, p_{\beta | \alpha}$.

To demonstrate, we fix some $\{a,b,c\} = \alpha \in \Gamma_1$, then we consider $\beta \in \Gamma_1$, $\beta \neq \alpha$.  
There is only one case to consider; that is, when $\beta=\{a,b,d\}$.  This scenario can be described pictorially as
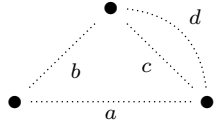
\[ \xymatrix{ & \bullet \ar@{.}[dr]_c \ar@{.}[dl]^b \ar@/^1pc/@{.}[dr]^d & \\ \bullet \ar@{.}[rr]_a & & \bullet }  \]
In this case, $Y_\alpha$ and $Y_\beta$ are dependent since knowing that edges $a$ and $b$ are already in a triangular formation affects the probability that $a$, $b$, and $d$ are in a triangular formation.  In fact, in this case we have $p_{\beta|\alpha} = \binom{n}{2}^{-1}$, since there is only one location allowed for edge $d$, i.e., it must coincide with edge $c$.  

Also, the terms $C_{2,2}$ and $C_{3,3}$ are the sum of two distinct forms of overlapping, and each has a different corresponding conditional probability.  Thus we subdivide these terms into $C_{2,2} = C_{2,2}^1 + C_{2,2}^2$ and $C_{3,3} = C_{3,3}^1 + C_{3,3}^2$.  
\begin{lemma}\label{Cp}
Let $C_{2,2}^1$ denote the number of indices $\alpha, \beta \in \Gamma_2$ which do \emph{not} share a triangle.  Let $C_{2,2}^2$ denote the number of indices $\alpha, \beta \in \Gamma_2$ which \emph{do} share a triangle.  Let $C_{3,3}^1$ denote the number of indices $\alpha, \beta \in \Gamma_3$ which share both triangles.  Let $C_{3,3}^2$ denote the number of indices $\alpha, \beta \in \Gamma_3$ which share a triangle and a single edge.  Then we have
\[
\begin{array}{llll}
C_{2,2}^1 & := \binom{\ks}{2}\binom{\kt}{2}, & p_{2|2}^1 & = \binom{n}{3}^{-1}, \\
C_{2,2}^2 & := \binom{\ks}{3}\binom{3}{1}\kt, & p_{2|2}^2 & = \binom{n}{2}^{-1}, \\
C_{3,3}^1 & := \binom{\kt}{2}\binom{\ks}{2}, & p_{3|3}^1 & = \binom{n}{2}^{-1}, \\
C_{3,3}^2 & := \ks \binom{\kt}{3}\binom{3}{1}, & p_{3|3}^2 & = \binom{n}{3}^{-1}.
\end{array}
\]
\end{lemma}

\begin{lemma}
We have 
\[ p_{1|1} = \binom{n}{2}^{-1}, \]
\[ p_{2|1} = p_{3|2} = p_{4|3} = p_{4|4} = \binom{n}{3}^{-1}, \]
and the rest are either specified in Lemma~\ref{Cp} or are 0.
\end{lemma}

In order to more easily state the bound for $b_2$, we make a final definition, which is for a collection of constants $C_{i,j}^\ast$, $i,j=1,\ldots, 7$, $i\leq j$.  
\begin{defi}
We define
\[ C_{2,2}^\ast := C_{2,2}^1 p_{2}p_{2|2}^1 + C_{2,2}^2 p_2 p_{2|2}^2, \]
\[ C_{3,3}^\ast := C_{3,3}^1 p_{3}p_{3|3}^1 + C_{3,3}^2 p_3 p_{3|3}^2. \]

For $i\leq j$, $i,j=1,\ldots, 7$, excluding the cases $i=j=2$ and $i=j=3$, we define
\[C_{i,j}^\ast := C_{i,j} p_i p_{j | i}.\]
\end{defi}

\begin{prop}
We have
\[ b_2 = \sum_{i= 1}^7 \sum_{j=i}^7 C_{i,j}^\ast  \sim O\left( \frac{(\ks+\kt)^4}{n^6}\right).
 \]
\end{prop}






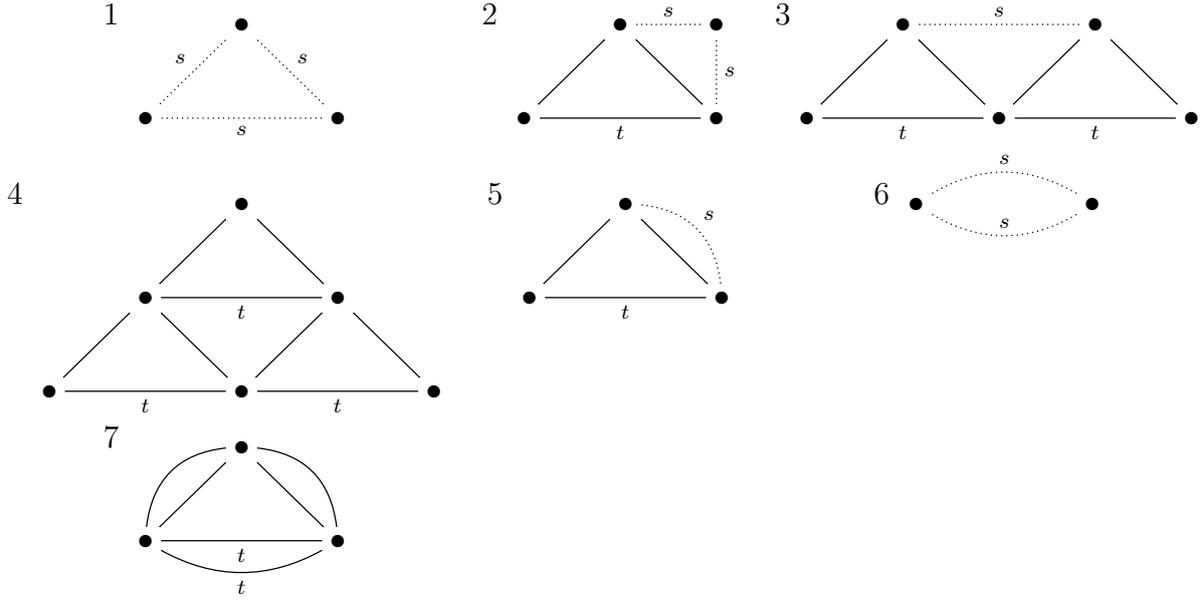
\begin{figure}
\begin{tabular}{ccc}
1 \xymatrix{ & \bullet \ar@{.}[dr]^s \ar@{.}[dl]_s& \\ \bullet \ar@{.}[rr]_s & & \bullet } &
2 \xymatrix{ & \bullet \ar@{-}[dr] \ar@{-}[dl] \ar@{.}[r]^s & \bullet \ar@{.}[d]^s \\ \bullet \ar@{-}[rr]_t & & \bullet }  &
3\xymatrix{ & \bullet \ar@{-}[dr] \ar@{-}[dl] \ar@{.}[rr]^s & & \bullet \ar@{-}[dr] \ar@{-}[dl] \\ \bullet \ar@{-}[rr]_t & & \bullet \ar@{-}[rr]_t & & \bullet }  \\
4 \xymatrix{& & \bullet \ar@{-}[dl] \ar@{-}[dr] \\ & \bullet \ar@{-}[dr] \ar@{-}[dl] \ar@{-}[rr]_t & & \bullet \ar@{-}[dr] \ar@{-}[dl] \\ \bullet \ar@{-}[rr]_t & & \bullet \ar@{-}[rr]_t & & \bullet }  &
5 \xymatrix{ & \bullet \ar@{-}[dr] \ar@{-}[dl] \ar@/^1pc/@{.}[dr]^s & \\ \bullet \ar@{-}[rr]_t & & \bullet }  &
6 \xymatrix{ \bullet \ar@/^1pc/@{.}[rr]^s\ar@/_1pc/@{.}[rr]^s & & \bullet} \\
7 \xymatrix{ & \bullet \ar@{-}[dr] \ar@{-}[dl]  \ar@/^1pc/@{-}[dr] \ar@/_1pc/@{-}[dl]  \\ \bullet \ar@{-}[rr]_t \ar@/_1pc/@{-}[rr]_t & & \bullet } 
\end{tabular}
\caption{Pictoral representation of the bad events in a random graph model}
\label{bad events}
\end{figure}

Next, we calculate the sizes of each set $\Gamma_i$, $i=1,\ldots, 7$. 

\begin{lemma}\label{Gs}
\[
\begin{array}{ll}
|\Gamma_1| & = \binom{k_s}{3},  \\
|\Gamma_2| & = \binom{k_s}{2}k_t, \\
|\Gamma_3| & = k_s \binom{k_t}{2}, \\
|\Gamma_4| & = \binom{k_t}{3}, \\
|\Gamma_5| & = k_s\, k_t, \\
|\Gamma_6| & = \binom{k_s}{2}, \\
|\Gamma_7| & = \binom{k_t}{2}. \\
\end{array}
\]
\end{lemma}

\begin{lemma}\label{lambda}
Let 
\[\lambda_j := \e X_j = |\Gamma_j|\, p_j,  \qquad \ j=1,2,\ldots, 7.\]  
In particular, we have
\[
\begin{array}{ccccccc}
\lambda_1 & = & \frac{\binom{k_s}{3} \binom{n}{3}}{\binom{n}{2}^3} &  \sim &\frac{2}{9}\frac{k_s^3}{n^3} &=& O\left( \frac{k_s^3}{n^3}\right) \\
\lambda_2 & = & \frac{\binom{k_s}{2} k_t \binom{n}{4} \binom{4}{3} \cdot 3}{\binom{n}{3} \binom{n}{2}^2} &\sim & 6\, \frac{k_s^2 k_t}{n^3} & = & O\left(\frac{k_s^2 k_t}{n^3}\right) \\
\lambda_3 & = & \frac{\binom{k_t}{2} k_s \binom{n}{5} \binom{5}{3} \binom{3}{2}}{\binom{n}{3}^2 \binom{n}{2}} &\sim &18\, \frac{k_t^2 k_s}{n^3} & = & O\left(\frac{k_t^2k_s}{n^3}\right) \\
\lambda_4 & = & \frac{\binom{k_t}{3} \binom{n}{6} \binom{6}{3} \binom{3}{2}^2}{\binom{n}{3}^3} & \sim & 9\, \frac{k_t^3}{n^3} & = & O\left( \frac{k_t^3}{n^3} \right) \\
\lambda_5 & = & \frac{k_t k_s \binom{n}{3}\cdot 3}{\binom{n}{3}\binom{n}{2}} & \sim & 6\, \frac{k_t k_s}{n^2} & = & O\left( \frac{k_t k_s}{n^2}\right) \\
\lambda_6 & = & \frac{\binom{k_s}{2} \binom{n}{2}}{\binom{n}{2}^2} & \sim & \frac{k_s^2}{n^2} & = & O\left(\frac{k_s^2}{n^2}\right) \\
\lambda_7 & = & \frac{\binom{k_t}{2} \binom{n}{3}}{\binom{n}{3}^2} & \sim & 3\, \frac{k_t^2}{n^3} & = & O\left(\frac{k_t^2}{n^3}\right).
\end{array}
\]
\end{lemma}


\begin{lemma} \label{W lemma}
Let $W = \sum_{j=1}^7 X_j$, and define $\lambda := \e W$.  
Suppose $Z$ is an independent Poisson random variable with expected value $\lambda$.  Then 
\[ d_{TV}(\mathcal{L}(W), \mathcal{L}(Z)) \leq  2 (b_1 + b_2), \]
where $b_1$ and $b_2$ are defined by Equation~\eqref{b1} and Equation~\eqref{b2}, respectively.
In addition, we have
\begin{equation}\label{eqW}
 (1-e^{-\lambda}) - d_0 \leq P(W>0) \leq (1-e^{-\lambda}) + d_{0}, 
 \end{equation}
where $d_0 = \min(1,\lambda^{-1})(b_1+b_2)$.  

\end{lemma}

\section{Applications}

\subsection{The number of graphs in $\mathcal{G}(n,\ks,\kt)$}

Let us start by comparing the two simpler graph models $\mathcal{G}(n,m)$ and $\mathcal{M}(n,m)$. We slightly abuse notation to let $|\mathcal{G}(n,m)|$ and $|\mathcal{M}(n,m)|$ denote the number of possible graphs in these respective probability spaces. We have
\[ |\mathcal{G}(n,m)| = \binom{\binom{n}{2}}{m}, \qquad |\mathcal{M}(n,m)| = \frac{\binom{n}{2}^m}{m!}. \]
Let $\P_1$ and $\P_2$ denote the probability measures of $\mathcal{G}(n,m)$ and $\mathcal{M}(n,m)$, respectively.  
There exists a coupling of $\P_1$ and $\P_2$ so that the random graph model either generates the same graph in both $\mathcal{G}(n,m)$ and $\mathcal{M}(n,m)$, or generates a graph in $\mathcal{M}(n,m) \setminus \mathcal{G}(n,m)$.  
The total variation distance in this case is simply 
\begin{align*}
d_{TV}(\P_1, \P_2) & =  \P(\text{graph model generated in $\mathcal{M}(n,m) \setminus \mathcal{G}(n,m)$}) \\
  & = 1 - \prod_{i=1}^{m-1} \left(1 - \frac{i}{\binom{n}{2}}\right) \sim 1 - e^{-m^2/n^2},
\end{align*}
whence
\[  d_{TV}(\P_1, \P_2) \to 0 \quad \iff \quad  m = o(n).  \]
Define $d := \prod_{i=1}^{m-1} \left(1 - \frac{i}{\binom{n}{2}}\right)$.  
Let $\mu$ denote the expected number of double edges.  We have $\mu = \frac{m(m-1)}{\binom{n}{2}}$, and so 
for all $n \geq 2$ and $ m \geq 1$, we have
\[ |\mathcal{M}(n,m)|e^{-\mu}(1-e^{\mu}(1-d)) \leq  |\mathcal{G}(n,m)| \leq |\mathcal{M}(n,m)|e^{-\mu}(1+e^{\mu}(1-d)). \]

\ignore{ 
\begin{proof}
This probability is given by 
\[ \frac{|\mathcal{G}(n,m)|}{|\mathcal{M}(n,m)|} = \frac{\binom{\binom{n}{2}}{m}}{\frac{\binom{n}{2}^m}{m!}} = \prod_{i=1}^{m-1} \left(1 - \frac{i}{\binom{n}{2}}\right). \]
This is precisely the same calculation as in the well--known birthday paradox, see for example \cite{Feller1}.  
\end{proof}
}

Now we generalize to $\mathcal{G}(n,\ks, \kt)$ and $\mathcal{M}(n,\ks,\kt)$.  Similar to the previous example, there exists a coupling between $\PG$ and $\PM$ so that the random graph model either generates the same graph in both $\mathcal{G}(n,\ks, \kt)$ and $\mathcal{M}(n,\ks,\kt)$, or generates a graph in $\mathcal{M}(n,\ks,\kt) \setminus \mathcal{G}(n,\ks,\kt)$.  We have
\[ |\mathcal{M}(n,\ks,\kt)| = \frac{\binom{n}{2}^{\ks}}{\ks!}\frac{\binom{n}{3}^{\kt}}{\kt!}. \]
Using Lemma~\ref{W lemma}, we estimate $|\mathcal{G}(n,\ks,\kt)|$ below.  
\begin{prop}\label{counting Gnk}
Let $\lambda$ and $d_0$ be defined as in Lemma~\ref{W lemma}.  Then for all $n \geq 3$, $\ks \geq 0$, $\kt \geq 0$, we have
\begin{equation}\label{counts Gnk}
 \frac{\binom{n}{2}^{\ks}}{\ks!}\frac{\binom{n}{3}^{\kt}}{\kt!} e^{-\lambda} (1 - e^{\lambda} d_0) \leq |\mathcal{G}(n,\ks,\kt)| \leq \frac{\binom{n}{2}^{\ks}}{\ks!}\frac{\binom{n}{3}^{\kt}}{\kt!} e^{-\lambda} (1 + e^{\lambda} d_0). 
 \end{equation}
For $\ks = O(n)$ and $\kt = O(n)$, asymptotically as $n\to\infty$, we have 
 \[ |\mathcal{G}(n,\ks,\kt)| \sim \frac{\binom{n}{2}^{\ks}}{\ks!}\frac{\binom{n}{3}^{\kt}}{\kt!} e^{-\lambda}. \]
\end{prop}
\begin{proof}
Simply rearrange Equation~\eqref{eqW} and note that $\lambda$ stays bounded if and only if $\ks = O(n)$ and $\kt = O(n)$ as $n\to\infty$, and $d_0 \to 0$ for these values of parameters. 
\end{proof}

\subsection{A confidence interval for $C(G)$}

With a bound on the total variation distance between $\PG$ and $\PM$, we can compute a confidence interval for $C(G)$ in $\mathcal{G}(n,\ks,\kt).$ 
Hence, define $C: \OmegaM \to [0,1],$ to be the random variable which specifies the clustering coefficients.  

For each pair of distinct edges $a, b$, where $1 \leq a < b \leq \ks,$ we let
\[ \hat{X}_{a,b} = 1(\text{edges $a$ and $b$ share exactly one node}). \]
The total number of connected triplets, which we denote by $W$, is given by $W = \sum_{a,b} \hat{X}_{a,b}$.  
In fact, the collection of random variables $\{\hat{X}_{a,b}\}_{1 \leq a<b \leq \ks}$ is an i.i.d.~sequence!  
Thus, $W$ is \emph{exactly} binomial with parameters $n = \binom{\ks}{2}$ and $p = \frac{n-2}{\binom{n}{2}}$, hence
\[ \e W = \binom{\ks}{2} \frac{n-2}{\binom{n}{2}}\]
and
\[ \text{Var}(W) = \binom{\ks}{2}\frac{n-2}{\binom{n}{2}}\left(1 - \frac{n-2}{\binom{n}{2}}\right). \]
Let $\lambda = \e W$ and $\sigma = \sqrt{\text{Var}(W)}.$  
For $\ks = O(\sqrt{n})$, $W$ is asymptotically Poisson distributed with parameter $\lambda$, and for $\ks/\sqrt{n} \to\infty$, $W$ is asymptotically normally distributed mean $\lambda$ and variance $\sigma^2$.  

In terms of the clustering coefficient $C$, we have
\[ C = \frac{3 \kt}{3 \kt + \displaystyle  W}. \]
A $(1-\alpha)$ level 2-sided confidence interval for $C$ is given by any numbers $L(C)$ and $U(C)$ that satisfy
\[ \PG(L(C) \leq C \leq U(C)) \geq 1-\alpha. \]
Rearranging, we have 
\[ \PG\left( 3 \kt \left(\frac{1}{U(C)}-1\right) \leq W \leq 3 \kt \left(\frac{1}{L(C)}-1\right) \right) \geq 1-\alpha.  \]

Let us consider the case when $W$ is asymptotically normally distributed, i.e., $\ks / \sqrt{n} \to\infty$.  
Letting $F$ and $\Phi$ denote the distribution functions of $(W-\lambda)$ and a normal random variable with mean 0 and variance $\sigma^2$, respectively, the Berry--Esseen theorem, as improved in~\cite{shevtsova2011absolute}, is given by
\[ \sup_x | F(x) - \Phi(x) |  \leq \frac{0.33554(\rho + 0.415\sigma^3)}{\sigma^3 \sqrt{n}}, \]
where 
\[ \rho = \e |\hat{X}_{a,b}-\e \hat{X}_{a,b}|^3 = \frac{n-2}{\binom{n}{2}}\left(1-\frac{n-2}{\binom{n}{2}}\right)\left( \left(1-\frac{n-2}{\binom{n}{2}}\right)^2 + \left(\frac{n-2}{\binom{n}{2}}\right)^2 \right).
\]

To form a confidence interval for $C(G)$ under measure $\P_{\mathcal{G}}$, we now work backwards from $W$ under measure $\P_{\mathcal{M}}$.  
First, we find a $(1-\beta)$ level 2-sided confidence interval for a random variable $Z$ from the standard normal distribution, where 
\[ \beta = \max(0, \alpha - \sup_x|F(x) - \Phi(x)| - d_{TV}(\PG,\PM)).
\]
Call the lower and upper bounds $L_\beta$ and $U_\beta$, respectively.  
Then we replace random variable $Z$ with random variable $(W-\lambda)/\sigma$, and rearrange to obtain 
\[ \P_{\mathcal{M}}(\sigma L_\beta + \lambda \leq W \leq \sigma U_\beta + \lambda) \geq 1-\alpha, \]
whence, 
\[ L(C) = \left( \frac{\sigma U_\beta + \lambda}{3 \kt} + 1\right)^{-1}, \quad U(C) = \left( \frac{\sigma L_\beta + \lambda}{3 \kt} + 1\right)^{-1}, \]
i.e., $[L(C), U(C)]$ is a $(1-\alpha)$ level confidence interval for $C$ in $\P_{\mathcal{G}}$.  

When $\lambda = O(1)$, $W$ is asymptotically Poisson distributed with parameter $\lambda$, and instead of a 2-sided confidence interval, we compute a 1-sided confidence interval.
A $(1-\alpha)$ level 1-sided confidence interval for $C$ in this case is given by any number $V(C)$ such that
\[ \PG(C \geq V(C)) \leq \alpha. \]
Rearranging, we see that the equivalent formulation in terms of $W$ is 
\[ \PM\left(W \leq \lambda + 3\kt \left(\frac{1}{\sigma V(C)}-1\right)\right) \leq \alpha. \]
Suppose $k_\alpha$ is the largest integer value such that
\[ \PM(W \leq k_\alpha) \leq \alpha. \]
Then we let $\beta = \max(0,\alpha-d_{TV}(\PG,\PM))$, and appeal to a table of Binomial probabilities\footnote{We can alternatively use LeCam's approximation, which states that $d_{TV}(\mathcal{L}(W), \text{Poisson}(\lambda)) \leq n p^2,$ but the resulting inversion of probabilities still requires a table of values.} and rearrange, to obtain
\[ V(C) = \frac{1}{\sigma} \left( \frac{k_\beta + \lambda}{3\kt} + 1\right)^{-1}. \]

\bibliographystyle{plain}
\bibliography{triangles}

\begin{thebibliography}{10}

\bibitem{ArratiaGoldstein}
Richard Arratia, Larry Goldstein, and Louis Gordon.
\newblock Two moments suffice for poisson approximations: the chen-stein
  method.
\newblock {\em The Annals of Probability}, pages 9--25, 1989.

\bibitem{bansal2008evolving}
Shweta Bansal, Shashank Khandelwal, and Lauren~Ancel Meyers.
\newblock Evolving clustered random networks.
\newblock {\em arXiv preprint arXiv:0808.0509}, 2008.

\bibitem{barabasi1999emergence}
Albert-L{\'a}szl{\'o} Barab{\'a}si and R{\'e}ka Albert.
\newblock Emergence of scaling in random networks.
\newblock {\em Science}, 286(5439):509--512, 1999.

\bibitem{barbour1992poisson}
Andrew~D Barbour, Lars Holst, and Svante Janson.
\newblock {\em Poisson approximation}.
\newblock Clarendon Press Oxford, 1992.

\bibitem{bloznelis2013degree}
Mindaugas Bloznelis.
\newblock Degree and clustering coefficient in sparse random intersection
  graphs.
\newblock {\em The Annals of Applied Probability}, 23(3):1254--1289, 2013.

\bibitem{bollobas1980probabilistic}
B{\'e}la Bollob{\'a}s.
\newblock A probabilistic proof of an asymptotic formula for the number of
  labelled regular graphs.
\newblock {\em European Journal of Combinatorics}, 1(4):311--316, 1980.

\bibitem{bollobas2003mathematical}
B{\'e}la Bollob{\'a}s and Oliver~M Riordan.
\newblock Mathematical results on scale-free random graphs.
\newblock {\em Handbook of graphs and networks: from the genome to the
  internet}, pages 1--34, 2003.

\bibitem{chen1975poisson}
Louis~HY Chen.
\newblock Poisson approximation for dependent trials.
\newblock {\em The Annals of Probability}, pages 534--545, 1975.

\bibitem{dall2002random}
Jesper Dall and Michael Christensen.
\newblock Random geometric graphs.
\newblock {\em Physical Review E}, 66(1):016121, 2002.

\bibitem{erdos1959random}
Paul Erd\H{o}s and Alfr\'{e}d R\'{e}nyi.
\newblock On random graphs {I}.
\newblock {\em Publicationes Mathematicae Debrecen}, 6:290--297, 1959.

\bibitem{gilbert1959random}
Edgar~N Gilbert.
\newblock Random graphs.
\newblock {\em The Annals of Mathematical Statistics}, pages 1141--1144, 1959.

\bibitem{holme2002growing}
Petter Holme and Beom~Jun Kim.
\newblock Growing scale-free networks with tunable clustering.
\newblock {\em Physical review E}, 65(2):026107, 2002.

\bibitem{klemm2002highly}
Konstantin Klemm and V\'{i}ctor~M Egu\'{i}luz.
\newblock Highly clustered scale-free networks.
\newblock {\em Physical Review E}, 65(3):036123, 2002.

\bibitem{newman2009random}
Mark E~J Newman.
\newblock Random graphs with clustering.
\newblock {\em Physical review letters}, 103(5):058701, 2009.

\bibitem{serrano2005tuning}
M~Angeles Serrano and Mari{\'a}n Bogun{\'a}.
\newblock Tuning clustering in random networks with arbitrary degree
  distributions.
\newblock {\em Physical Review E}, 72(3):036133, 2005.

\bibitem{shevtsova2011absolute}
Irina Shevtsova.
\newblock On the absolute constants in the berry-esseen type inequalities for
  identically distributed summands.
\newblock {\em arXiv preprint arXiv:1111.6554}, 2011.

\bibitem{watts1998collective}
Duncan~J Watts and Steven~H Strogatz.
\newblock Collective dynamics of `small-world' networks.
\newblock {\em Nature}, 393(6684):440--442, 1998.

\bibitem{yagan2009random}
Osman Yagan and Armand~M Makowski.
\newblock Random key graphs can they be small worlds?
\newblock In {\em Networks and Communications, 2009. NETCOM'09. First
  International Conference on}, pages 313--318. IEEE, 2009.

\end{thebibliography}
\end{document}